\documentclass[12pt]{article}

\usepackage{amsmath,amsthm}
\usepackage{amsfonts,amsrefs}
\usepackage{verbatim}
\usepackage{amssymb}
\usepackage{txfonts}
\usepackage{amscd}
\usepackage{bbm}
\usepackage{hyperref}
\usepackage{mathrsfs}
\usepackage{fouriernc}

\newtheorem*{maintheorem}{The Main Theorem}
\newtheorem{theorem}{Theorem}
\newtheorem{corollary}[theorem]{Corollary}
\newtheorem{lemma}[theorem]{Lemma}
\newtheorem{proposition}[theorem]{Proposition}
\newtheorem{definition}[theorem]{Definition}

\newtheorem{remark}[theorem]{Remark}

\newtheorem{observation}[theorem]{Observation}

\begin{document}

\title{Group algebra criteria for vanishing of cohomology}
\author{Uri Bader \& Piotr W. Nowak}

\maketitle

\begin{abstract}
Given a group satisfying sufficient finiteness properties, we discuss a group algebra criterion for vanishing of all its cohomology groups with unitary coefficients in a certain degree.
\end{abstract}

\section{Introduction}

In this manuscript we consider the group property of vanishing of all cohomology groups with unitary coefficients in a certain degree.
Recall that for finitely generated groups the vanishing of all {\em first} cohomology with unitary coefficients 
is equivalent to Kazhdan's property $(T)$.
Note that property $(T)$ is related to graph expansion properties and has numerous applications.
Recently, the vanishing of all {\em second} cohomology with unitary coefficients was brought to the front in \cite{chifre-etal},
where it was observed that for finitely presented groups this property implies Frobenius-stability.
Vanishing of all cohomology groups with unitary coefficients in {\em higher} degrees also gained recent interests, as it is related to the study of higher-dimensional expanders, see e.g. \cite{lubotzky}.
Spectral criteria for vanishing of all cohomology groups with unitary coefficients in a certain degree were given by Ballmann and Swiatkowski in
\cite{Ballmann-Swiatkowski}, following Garland \cite{Garland}, see also \cite{dymara-janushkevich, oppenheim1, oppenheim2} and \cite{oppenheim-gr}.
Here we are concerned with the question whether, given a group satisfying sufficient finiteness properties, the vanishing of all its 
cohomology groups with unitary coefficients in a certain degree could be 
detected in the level of the group algebra.

A source of inspiration for us is the work of Narutaka Ozawa,
who settled the above question in the first degree. Indeed, he showed in \cite{ozawa} that a finitely generated group $\Gamma$ has peoperty $(T)$ if and only if the 
equation $$\Delta^2=\epsilon\Delta+\sum_{i=1}^m x_i^*x_i$$
has a solution in the group ring $\mathbb{Q}\Gamma$ for some positive rational $\epsilon$, where $\Delta$ denotes the 
Laplacian in $\mathbb{Q}\Gamma$. This allowed computer-assisted methods to be used 
to prove property $(T)$ for some groups, including some classical ones such as $\operatorname{SL}_n(\mathbb{Z})$ for $n=3,4,5$, see \cite{netzer-thom,fujiwara-kabaya,kaluba-nowak},
as well as new ones such as $\operatorname{Aut}(F_n)$ for $n \ge 5$, see 
\cite{ kaluba-nowak-ozawa,kaluba-kielak-nowak}.
Our current contribution is an analogue of Ozawa's characterization 
in higher degrees. 
Before stating it, we recall the following standard fact.

\begin{proposition} \label{prop:D_n}
Let $\Gamma$ be a group.
Assume $\Gamma$ is acting with finite stabilizers by automorphisms on a contractible simplicial complex $X$.
Let $N$ be a natural such that for every $n\leq N+1$, the $n$'th skeleton of $X$ has finitely many $\Gamma$-orbits.
Denote this number of orbits by $k_n$.
Then for $n\leq N$ there exist group algebra matrices $D_n\in M_{k_n\times k_{n+1}}(\mathbb{Q}\Gamma)$
such that $D_nD_{n-1}=0$ and for every $\mathbb{Q}\Gamma$ module $V$, the cohomology groups $H^n(\Gamma,V)$ are isomorphic to the cohomology groups of the complex
\[ \cdots \to V^{k_{n-1}} \overset{D_{n-1}}{\longrightarrow} V^{k_n} \overset{D_{n}}{\longrightarrow} V^{k_{n+1}} \to \cdots \]
\end{proposition}
We give a proof of the above proposition in Section \ref{sec:main}.
Let $I\in M_{k_n}(\mathbb{Q}\Gamma)$ denote the identity matrix and for $x\in M_{k_n}(\mathbb{Q}\Gamma)$, let 
$x^*$ denote the matrix in $M_{k_n}(\mathbb{Q}\Gamma)$ obtained by transposing the matrix $x$ and inverting the 
group elements in its entries.
The following is the main theorem of this paper.

\begin{maintheorem} 
Let $\Gamma$ be a group.
Assume $\Gamma$ is acting with finite stabilizers by automorphisms on a contractible simplicial complex $X$.
Let $N$ be a natural such that for every $n\leq N+1$, the $n$'th skeleton of $X$ has finitely many $\Gamma$-orbits.
Denote this number of orbits by $k_n$.
For $n\leq N$ we let $D_n\in M_{k_n\times k_{n+1}}(\mathbb{Q}\Gamma)$ be the matrix over the rational group algebra given in
Proposition~\ref{prop:D_n}
and we set 
\[ \Delta_n=D_n^*D_n+D_{n-1}D_{n-1}^* \in M_{k_n}(\mathbb{Q}\Gamma). \]
Then, for a fixed $1\leq n\leq N$, the following are equivalent:
\begin{enumerate}
\item For every unitary representation $\rho$ of $\Gamma$, $H^n(\Gamma,\rho)=0$ and the topology of $H^{n+1}(\Gamma,\rho)$
is Hausdorff (that is $H^{n+1}(\Gamma,\rho)$ is reduced, see \S\ref{sec:hilbertchain} for the exact definition).
\item There exist a rational $\epsilon>0$ and elements $x_1,\ldots,x_m\in M_{k_n}(\mathbb{Q}\Gamma)$
such that 
\[ \Delta_n=\epsilon I+\sum_{i=1}^m x_i^*x_i. \]
\end{enumerate}
\end{maintheorem}

The result follows easily by combining the well understood theories of $*$-algebras, which we will recall in \S\ref{sec:algebras},
and of chain complexes of Hilbert spaces, which we will recall in \S\ref{sec:hilbertchain}.
We will discuss how to obtain The Main Theorem from these theories in \S\ref{sec:main},
giving its actual proof in \S\ref{subsec:proof}, after recalling the proof of Proposition~\ref{prop:D_n} in \S\ref{subsec:D_n}.
Finally, we will compare our result with \cite{ozawa} and discuss related characterizations of the property that $H^n(\Gamma,\rho)$ is reduced 
(that is, Hausdorff) for every unitary representation $\rho$ of $\Gamma$
in \S\ref{subsec:remarks}.

\section*{Acknowledgment}

This paper grew out of our discussion of a question by Alex Lubotzky.
We wish to thank Alex for asking us this question and for his interest in this work.
We thank Jan Dymara and Izhar Oppenheim for reading a first draft of this manuscript and sharing their insights with us,
and we thank Marek Kaluba for holding some computer experiments to our request. We also thank the referee for comments and corrections.
Finally, we thank the Banach center in Warsaw for hosting us during the semester on Analytic Group Theory in spring 2019.

UB is supported by the
ISF Moked 713510 grant number 2919/19.
PN is partially supported by the European Research Council (ERC) under the European
Union's Horizon 2020 research and innovation programme (grant agreement no. 677120-INDEX).
This work was partially supported by the grant 346300 for IMPAN from the Simons Foundation and the matching 2015-2019 Polish MNiSW fund

\section{Preliminaries on $*$-algebras} \label{sec:algebras}

In this section we review some basic facts in the theory of $*$-algebras.
By a $*$-algebra we mean a unital algebra over the field of rational numbers $\mathbb{Q}$ which is endowed with an involution, denoted $x\mapsto x^*$, which is anti-automorphic, i.e it is $\mathbb{Q}$-linear and satisfying $(xy)^*=y^*x^*$. 
All algebras considered here are unital and all algebra morphisms are morphisms of unital algebras.
When a scalar is considered as an element in an algebra, it is regarded as applied to the algebra unit.

\subsection{Archimedean algebras, following Cimpri\v{c}}

In this subsection we fix a $*$-algebra $A$.
We denote 
\[ A_h=\{x\in A\mid x=x^*\} \]
and 
\[ A_+=\left\{\sum_{i=1}^n x_i^*x_i \mid n\in \mathbb{N},~x_1,\ldots,x_n \in A\right\}. \]
Note that $A_h<A$ is a subvector space and $A_+ \subset A_h$ has the following properties:
\begin{align*}
x,y \in A_+ & \quad \Rightarrow \quad x+y \in A_+, \\
x\in A,~y\in A_+ & \quad \Rightarrow  \quad x^*yx \in A_+, \\
\alpha \in \mathbb{Q}_+, ~x\in A_+ & \quad \Rightarrow \quad \alpha x\in A_+. 
\end{align*}
The first two properties are immediate and the third follows as every positive rational is a sum of squares of rationals,
indeed $p/q=\sum_{i=1}^{pq} 1/q^2$.

We endow $A$ with the partial order: 
\[ x\leq y \quad \Longleftrightarrow \quad y-x \in A_+ \]
and, considering $\mathbb{Q}<A$, we define for $x\in A$,
\[ \|x\|=\sqrt{\inf \{\alpha\in \mathbb{Q}_+\mid x^*x \leq \alpha\}} \in [0,\infty]. \]
Here we use the conventions $\inf \emptyset=\infty$ and $\sqrt{\infty}=\infty$.
If $\|x\|<\infty$ (that is, there exists $\alpha\in \mathbb{Q}_+$ such that $\|x\|\leq \alpha$), we say that $x$ is bounded. 
We denote the collection of bounded elements in $A$ by $A_b$.
If $\|x\|=0$ we say that $x$ is infinitesimal.
We denote the collection of infinitesimal elements in $A$ by $A_i$.

\begin{theorem}[{Cimpri\v{c}, see \cite[Theorem 3.2, Corollary 3.3]{cimpric}}] \label{thm:cimpric}
Assume $-1\notin A_+$. Then for $x,y\in A_b$ and $\alpha\in \mathbb{Q}$ we have
\begin{enumerate}
\item $\|1\|=1$,
\item $\|\alpha x\|=|\alpha|\|x\|$,
\item $\|x\|=\|x^*\|$,
\item $\|x+y\|\leq \|x\|+\|y\|$,
\item $\|xy\|\leq \|x\|\|y\|$,
\item $\|x^*x\|=\|x\|^2$,
\item $\|x^*x\| \leq \|x^*x+y^*y\|$.
\end{enumerate}
In particular, $A_b<A$ is a $*$-subalgebra, $A_i$ is a two sided $*$-ideal in $A_b$,
$\|\cdot\|$ is a seminorm on $A_b$ which descends to a norm on $A_b/A_i$ and the associated completion of $A_b/A_i$ is a real $C^*$-algebra.
\end{theorem}

As the discussion in \cite{cimpric} is in a slightly more general setting than ours, while the proof of the theorem is rather elementary, we will reproduce the proof 
for the reader's convenience.
We first state a useful lemma.

\begin{lemma}[cf. {\cite[Lemma~3.1]{cimpric}}] \label{lem:cimpric}
For $x\in A_h$ and $\alpha \in \mathbb{Q}_+$, 
$x^2\leq \alpha^2$ implies $x\leq \alpha$.
In particular,
$\|x\|< \alpha$ implies $x < \alpha$.
\end{lemma}

\begin{proof}
Assume $\alpha \in \mathbb{Q}_+$ and $x^2\leq \alpha^2$. Then $2\alpha (\alpha-x)=(\alpha-x)^2+(\alpha^2-x^2) \geq 0$,
thus $x\leq \alpha$.
Assume now
$\|x\|<\alpha$. Fix $\beta\in \mathbb{Q}_+$ such that $\|x\|<\beta<\alpha$. Then $x^2\leq \beta^2$, thus
$x\leq \beta <\alpha$.
\end{proof}

\begin{proof}[{Proof of Theorem~\ref{thm:cimpric}}]
(1) follows from the assumption $-1\notin A_+$.
It is clear that for $\alpha\in \mathbb{Q}_+$, $\|\alpha x\|=\alpha \|x\|$ and that $\|-x\|=\|x\|$, thus (2) follows.
For (3) it is enough to show that $\|x^*\|\leq \|x\|$.
We fix $\alpha\in \mathbb{Q}_+$ such that $x^*x\leq \alpha$ and argue to show that $xx^*\leq \alpha$.
We note that 
$x(\alpha-x^*x)x^*\in A_+$, as $\alpha-x^*x$ is, hence we have
\[ \left(\alpha/2\right)^2-\left(\alpha/2-xx^*\right)^2=x(\alpha-x^*x)x^* \geq 0, \]
thus $\left(\alpha-x^*x\right)^2 \leq \left(\alpha/2\right)^2$.
By Lemma~\ref{lem:cimpric} we conclude that $xx^*- \alpha/2 \leq \alpha/2$,
hence indeed $xx^* \leq \alpha$.
The proof of (4) is postponed until later.
To prove (5) we fix $\alpha,\beta\in \mathbb{Q}_+$ such that $x^*x \leq \alpha$ and $y^*y\leq \beta$
and argue to show that $(xy)^*(xy)\leq \alpha\beta$.
We note that $y^*(\alpha-x^*x) y$ is in $A_+$, as $\alpha-x^*x$ is,
hence we have
\[ (xy)^*(xy)=\alpha y^*y-y^*(\alpha-x^*x) y \leq \alpha y^*y \leq \alpha\beta, \]
thus indeed $(xy)^*(xy)\leq \alpha\beta$. 
We now prove (6). 
By (5) and (3) we have $\|x^*x\|\leq \|x\|^2$, thus we need to show $\|x\|^2\leq \|x^*x\|$.
We fix $\alpha \in \mathbb{Q}_+$ such that $\|x^*x\|<\alpha$ and argue to show that $\|x\|^2\leq \alpha$.
By Lemma~\ref{lem:cimpric} we have that
$x^*x< \alpha$ and by the definition of $\|\cdot\|$ we get that $\|x\|\leq \sqrt{\alpha}$.
Thus indeed, $\|x\|^2\leq \alpha$.
To see (7) we fix $\alpha \in \mathbb{Q}_+$ such that $\|x^*x+y^*y\|<\alpha$ and argue to show that 
$\|x^*x\|<\alpha$.
By Lemma~\ref{lem:cimpric} we have
$x^*x\leq x^*x+y^*y< \alpha$, thus indeed, by Lemma~\ref{lem:cimpric} again, $\|x^*x\|<\alpha$.
Finally, we are back to (4).
We fix $\alpha,\beta\in \mathbb{Q}_+$ such that $\|x\| < \alpha$ and $\|y\|< \beta$
and argue to show that $(x+y)^*(x+y)\leq (\alpha+\beta)^2$, thus $\|x+y\|\leq \alpha+\beta$.
We have by (5) and (3) that $\|x^*y\|,\|y^*x\| < \alpha\beta$, thus 
$(x^*y)^*(x^*y),(y^*x)^*(y^*x) \leq (\alpha\beta)^2$, and we get
\[ 4\alpha^2\beta^2-(x^*y+y^*x)^2= \]
\[
2\left((\alpha\beta)^2-(x^*y)^*(x^*y)\right)+
2\left((\alpha\beta)^2-(y^*x)^*(y^*x)\right)+
(x^*y-y^*x)^*(x^*y-y^*x)  \leq 0. \]
Therefore 
$(x^*y+y^*x)^2 \leq 4\alpha^2\beta^2$
By Lemma~\ref{lem:cimpric} we get that
$x^*y+y^*x \leq 2\alpha\beta$. Therefore 
\[ (\alpha+\beta)^2-(x+y)^*(x+y)=(\alpha^2-x^*x)+(\beta^2-y^*y)+ \left(2\alpha\beta-(x^*y+y^*x)\right) \geq 0, \]
thus indeed $(x+y)^*(x+y)\leq (\alpha+\beta)^2$. 
This proves (4).

That $A_b$ is a $*$-subalgebra and $A_i$ is an ideal in it follow at once from (1)-(5)
and it is clear that $\|\cdot\|$ descends to a norm on $A_b/A_i$.
The completion of the scalars, $\mathbb{Q}<A_b/A_i$, with respect to this norm is isomorphic to $\mathbb{R}$, thus the completion of
$A_b/A_i$ becomes a real vector space, indeed a Banach space.
By (5) and (3) it is a Banach $*$-algebra
and (6)-(7) are the defining axioms for real $C^*$-algebras among Banach $*$-algebras, see \cite{Ingelstam} or \cite{Palmer}.
\end{proof}

\begin{definition}
We say that the $*$-algebra $A$ is archimedean if $-1$ is not in $A_+$ and all of its elements are bounded, that is $A=A_b$.
In this case we denote by $C^*_{\mathbb{R}}(A)$ the completion of $A/A_i$ and call it the real $C^*$-completion of $A$.
We endow the complexification, $\mathbb{C}\otimes_\mathbb{R} C^*_{\mathbb{R}}(A)$, with its canonical $C^*$-norm and $*$-operator,
and denote it $C^*(A)$. We call $C^*(A)$ the complex $C^*$-completion, or merely the $C^*$-completion of $A$.
\end{definition}
We will need the following 
\begin{observation}
An element $x\in A$ is called a partial isometry if $xx^*x=x$.
If $x$ is a partial isometry then $x$ is bounded. Indeed,
noting that $x^*xx^*x=x^*x$, we have
\[ x^*x \leq x^*x+(1-x^*x)^*(1-x^*x) = 1,\]
thus $\|x\|\leq 1$.
\end{observation}

A state on the $*$-algebra $A$ is a linear functional $s:A\to\mathbb{C}$ satisfying $s(A_+)\subset [0,\infty)$ and $s(1)=1$.
Note that if $A$ is archimedean then it has a state, as $C^*(A)$ has a state which we can pull back to $A$ via the natural map $A\to C^*(A)$.
Conversely, the existence of a state on $A$ is a checkable condition for guaranteeing $-1\notin A_+$.

\begin{corollary} \label{cor:archicriterion}
If a $*$-algebra is generated by partial isometries and it has a state then it is archimedean.
\end{corollary}

Given $*$-algebras $A$ and $B$ we endow the algebras $A\oplus B$ and $A\otimes B$ with the $*$-operations given
by $(x,y)^*=(x^*,y^*)$ and $(x\otimes y)^*=x^* \otimes y^*$ correspondingly.

\begin{lemma} \label{lem:tensor}
If $A$ and $B$ are archimedean $*$-algebras than so are the $*$-algebras $A\oplus B$ and $A\otimes B$.
\end{lemma}

\begin{proof}
Assume $A$ and $B$ are archimedean.
Letting $s$ and $t$ be states on $A$ and $B$ correspondingly, observe that
$\frac{1}{2}s+\frac{1}{2}t$ and $s\cdot t$ are corresponding states on $A\oplus B$ and $A\otimes B$, thus $(-1,-1) \notin (A\oplus B)_+$ and 
$-1 \otimes 1 \notin (A\otimes B)_+$.
Observe also that both
\[ A\oplus \{0\} = A_b \oplus \{0\} < (A\oplus B)_b \]
and 
\[ \{0\} \oplus B = \{0\} \oplus B_b < (A\oplus B)_b, \]
thus by Theorem~\ref{thm:cimpric}
\[ A\oplus B < (A\oplus B)_b < A\oplus B, \]
and we conclude  that $(A\oplus B)_b = A\oplus B$.
Similarly,
\[ A\otimes \mathbb{Q} = A_b \otimes \mathbb{Q} < (A\otimes B)_b \]
and 
\[ \mathbb{Q} \otimes B = \mathbb{Q} \otimes B_b < (A\otimes B)_b, \]
thus by Theorem~\ref{thm:cimpric}
\[ A\otimes B < (A\otimes B)_b < A\otimes B, \]
and we conclude  that $(A\otimes B)_b = A\otimes B$.
Thus indeed both $A\oplus B$ and $A\otimes B$ are archimedean.
\end{proof}

\subsection{On the $C^*$-completion of archimedean $*$-algebras}

A morphism of $*$-algebras is an algebra morphism which preserves the corresponding $*$-operation.
Note that every such morphism is order preserving, hence also norm non-increasing.
Note also that $C^*$-algebras are always archimedean and the order and norm defined in the previous subsection coincide with the $C^*$-theoretic order and norm.
Given an archimedean $*$-algebra $A$, any morphism $A\to B$ into a $C^*$-algeba $B$ must be trivial on $A_i$, and the corresponding 
morphism $A/A_i \to B$ extends to $C^*_{\mathbb{R}}(A)\to B$, by the fact that $A\to B$ is norm non-increasing.
Further, $C^*_{\mathbb{R}}(A)\to B$ extends to the complexification $C^*(A)\to B$, by its universal property.
The following is an immediate consequence of this list of observations.

\begin{proposition} \label{prop:c*universal}
The association $A \mapsto C^*(A)$ is a functor from the category of 
arcimedean $*$-algebras to the category of $C^*$-algebras which is left adjoint to the forgetful functor from the category of C*-algebras to the category of 
arcimedean $*$-algebras.
\end{proposition}

The following two propositions, comparing properties of elements of $A$ with corresponding properties of their images in 
$C^*(A)$, are due Schm\"{u}dgen.

\begin{proposition}[{cf. \cite[Proposition~14]{schmudgen}}] \label{prop:poscriterion}
Let $A$ be an archimedean $*$-algebra and $x\in A_h$. Denote by $\bar{x}$ the image of $x$ in $C^*(A)$.
Then $\bar{x}\geq 0$ iff for every $\alpha\in \mathbb{Q}_+$, $x> -\alpha$.
\end{proposition}

\begin{proof}
Assume that for every $\alpha\in \mathbb{Q}_+$, $x> -\alpha$. Then for every $\alpha\in \mathbb{Q}_+$, $\bar{x}> -\alpha$.
It follows that the spectrum of $\bar{x}$ is contained in $[0,\infty)$, thus $\bar{x}\geq 0$.
Assume now that $\bar{x}\geq 0$ and fix $\alpha\in \mathbb{Q}_+$.
Fix $\beta\in \mathbb{Q}_+$ such that $\|\bar{x}\|\leq \beta$ and observe that the spectrum of $\bar{x}$ is contained in $[0,\beta]$, thus so is the spectrum of $\beta-\bar{x}$.
It follows that $\|\beta-\bar{x}\|\leq \beta$, thus
\[ \|\beta-x\|=\|\beta-\bar{x}\|\leq \beta<\beta+\alpha. \]
By Lemma~\ref{lem:cimpric} we get that $\beta-x< \beta+\alpha$,
thus indeed $x>-\alpha$.
\end{proof}

\begin{proposition}[{cf. \cite[Proposition~15]{schmudgen}}]
Let $A$ be an archimedean $*$-algebra and $x\in A_+$ a positive element.
Denote by $\bar{x}$ the image of $x$ in $C^*(A)$.
Then $\bar{x}$ is invertible in $C^*(A)$ iff
there exists $\in \mathbb{Q}_+$ such that $x\geq \alpha$.
\end{proposition}

\begin{proof}
By Proposition~\ref{prop:poscriterion}
there exists $\alpha \in \mathbb{Q}_+$ such that $x\geq \alpha$ iff
there exists $\alpha \in \mathbb{Q}_+$ such that $\bar{x}\geq \alpha$.
As $\bar{x}\geq 0$, the latter condition is equivalent to $0$ not being in the spectrum of $\bar{x}$,
and by functional calculus this is equivalent to the invertibility of $\bar{x}$.
\end{proof}

We end up this section by stating a corollary that will be used later.
This corollary is merely a reformulation of the last proposition, taken into account the universality of $C^*(A)$ discussed in 
Proposition~\ref{prop:c*universal}, and the fact that every $C^*$-algebra could be embedded in the $C^*$-algebra of bounded operators 
on a Hilbert space.

\begin{corollary} \label{cor:invcriterion}
Let $A$ be an archimedean $*$-algebra and $x\in A_+$ a positive element.
There exists $\alpha \in \mathbb{Q}_+$ such that $x\geq \alpha$ iff $\rho(x)$ is invertible
for every $*$-representation $\rho:A\to B(V)$, where $V$ is a Hilbert space and $B(V)$ the $*$-algebra of bounded operators
on $H$.
\end{corollary}

%%%%

\subsection{Matrix algebras over $*$-algebras}

In this subsection we discuss well understood results regarding representations of matrix algebras over unital $*$-algebras.

We fix a natural number $k$.
The matrix algebra $M_k(\mathbb{Q})$ is endowed with the standard $*$-operation, namely the transposition.
Given a unital $*$-algebra $A$, 
we identify as usual the tensor product $*$-algebra $M_k(\mathbb{Q})\otimes A$ with the matrix algebra $M_k(A)$.
Under this identification, $A\simeq \mathbb{Q}\otimes A$ is identified with the scalar matrices in $M_k(A)$
and $M_k(\mathbb{Q})\simeq M_k(\mathbb{Q})\otimes \mathbb{Q}$ is identified with the $\mathbb{Q}$-valued matrices.
In particular, $M_k(A)$ is generated as an algebra by its commuting subalgebras $A$ and $M_k(\mathbb{Q})$.
Under this identification the tensor product $*$-operation on $M_k(\mathbb{Q})\otimes A$ is identified with the operation
\[ (x_{ij})^*=(x_{ji}^*) \]
on $M_k(A)$.

We observe that for every $*$-representation $\rho:A\to B(U)$, where $U$ is a Hilbert space and $B(U)$ is the $*$-algebra of bounded operators
on $U$, 
there exists a naturally associated $*$-representation $M_k(\rho):M_k(A) \to B(U^n)$.
Under our identification $M_k(A) \simeq M_k(\mathbb{Q})\otimes A$,
and upon identifying $U^n\simeq \mathbb{Q}^n\otimes U$, $M_k(\rho)$ is identified with $\mbox{id}\otimes \rho$.
The next proposition explains that all Hilbert $*$-representations of $M_k(A)$ are constructed in this way.

\begin{proposition}\label{proposition: reps of M_k(A) are direct sums}
For every $*$-representation $\theta:M_k(A)\to B(V)$, where $V$ is a Hilbert space and $B(V)$ the $*$-algebra of bounded operators
on $V$, 
there exits a Hilbert space $U$ and a $*$-representation $\rho:A\to B(U)$ such that $\theta$ is isomorphic to $M_k(\rho)$ as Hilbert $*$-representations,
that is there exists an isometric isomorphism $\phi:V\to U^n$ such that for every $x\in M_k(A)$, $M_k(\rho)(x)=\phi(\theta(x))$.
\end{proposition}

\begin{proof}
Let $e_{ij}\in M_k(\mathbb{Q})$ be the standard basis matrix and let $p_i=e_{ii}$. Let $V_i=p_iV$ and $U=V_1$.
Note that, by the commutation of $A$ and $M_k(\mathbb{Q})$, each $V_i$ is a $*$-representation of $A$.
Applying permutation matrices one obtains that they are all isomorphic $A$-representations. In particular, they are all isomorphic, as $A$-representations, to $U$. 
Since $\sum p_i=1$ and $p_ip_j=\delta_i^j p_i,$
we obtain $$V=\oplus V_i \simeq U^n.$$
This is an isomorphism as $A$-representations and as $M_n(\mathbb{Q})$-representation, thus this is an isomorphism as $M_k(A)$-representations.
\end{proof}

Noting that the normalized trace is a state on $M_k(\mathbb{Q})$ and that the matrices $e_{ij}$ are partial isometries that generate $M_k(\mathbb{Q})$
we get by Corollary~\ref{cor:archicriterion} that $M_k(\mathbb{Q})$ is arcimedean.
As $M_k(A)\simeq M_k(\mathbb{Q}) \otimes A$, we get by Lemma~\ref{lem:tensor}
the following.

\begin{corollary} \label{cor:Mncarch}
If $A$ is an archimedean $*$-algebra then so is $M_k(A)$.
\end{corollary}

\subsection{$x^*x$ as sum of squares for a non-square matrix}

This subsection is devoted to the proof of an observation related to non-square matrices that will be useful for us later on.
We let $A$ be a $*$-algebra and, for naturals $k$ and $k'$, we consider the obvious $*$-operation from $M_{k\times k'}(A)$ 
to $M_{k'\times k}(A)$ given by $(x_{ij})^*=(x_{ji}^*)$.

\begin{lemma} \label{lem:squarematrices}
For $x\in M_{k\times k'}(A)$, we have $x^*x\in M_{k'}(A)_+$.
\end{lemma}

\begin{proof}
For every naturals $n$ and $i\leq n$, we let $e^n_i\in M_{n\times 1}(\mathbb{Q})$ be the standard $i$'th basis vector.
We note that $(e^n_1)^*e^n_1=1\in \mathbb{Q}$ and that $\sum_{i=1}^n e^n_i(e^n_i)^*=1\in M_n(\mathbb{Q})$.
We thus have
\[ x^*x=x^*\left(\sum_{i=1}^{k} e^{k}_i(e^{k}_i)^*\right)x=\sum_{i=1}^k x^*e^k_i(e^k_i)^*x=
\sum_{i=1}^k x^*e^k_i(e^{k'}_1)^*e^{k'}_1(e^k_i)^*x= \sum_{i=1}^k y_i^*y_i, \]
where $y_i=e^{k'}_1(e^k_i)^*x \in M_{k'}(A)$.
\end{proof}

\subsection{Matrix algebras over group algebras}

In this subsection we fix a group $\Gamma$ and specialize the discussion to the group algebra $A=\mathbb{Q}\Gamma$,
that is the algebra of finitely supported rational valued functions on $\Gamma$.
We endow $\mathbb{Q}\Gamma$ with the standard $*$-operation given by $f^*(g)=f(g^{-1})$.
This algebra is generated by the delta functions, which are partial isometries, and the evaluation at the identity element is a state.
Hence, by Lemma~\ref{lem:tensor}, $\mathbb{Q}\Gamma$ is archimedean.

Viewing $\Gamma$ as a subgroup of $\mathbb{Q}\Gamma$ via the identification of elements of $\Gamma$ with the corresponding delta functions 
in $\mathbb{Q}\Gamma$, we note that every $*$-representations of $\mathbb{Q}\Gamma$ on Hilbert spaces restricts to a unitary representation of $\Gamma$, while every unitary representation of $\Gamma$ extends by linearity to a $*$-representation of $\mathbb{Q}\Gamma$.
We thus get a natural correspondence between unitary representations of $\Gamma$ and Hilbert $*$-representations of $\mathbb{Q}\Gamma$.
We thus observe that the $C^*$-completion of $\mathbb{Q}\Gamma$, $C^*(\mathbb{Q}\Gamma)$, is nothing but $C^*\Gamma$, the maximal $C^*$-algebra of $\Gamma$.

We now fix a natural $k$ and discuss the matrix algebra $M_k(\mathbb{Q}\Gamma)$.
By Corollary~\ref{cor:Mncarch} we have that $M_k(\mathbb{Q}\Gamma)$ is archimedean.
By Proposition~\ref{proposition: reps of M_k(A) are direct sums}
we see that Hilbert $*$-representations of $M_k(\mathbb{Q}\Gamma)$ are all induced from Hilbert $*$-representations of $\mathbb{Q}\Gamma$,
thus from unitary representations of $\Gamma$.
We end up this section with the following corollary, which is a specialization of Corollary~\ref{cor:invcriterion}
in view of the above discussion.

\begin{corollary} \label{cor:unitaryinvcriterion}
Let $x\in M_k(\mathbb{Q}\Gamma)$ be a positive element.
There exists $0< \alpha\in \mathbb{Q}_+$ such that $x\geq \alpha$ iff the image of $x$ under $M_k(\rho)$ is invertible
for every unitary representation $\rho$ of $\Gamma$.
\end{corollary}

\section{Preliminaries on Hilbert chain complexes} \label{sec:hilbertchain}

In this section we review some classical observations regarding 
Laplace operators on complexes of Hilbert spaces.
We assume having a chain complex of Hilbert spaces and bounded maps
\[ \cdots \to C_{n-1} \overset{d_{n-1}}{\longrightarrow} C_n \overset{d_{n}}{\longrightarrow} C_{n+1} \to \cdots, \]
for short, $(C_\bullet,d_\bullet)$.
We let $\partial_n=d^*_{n-1}:C_n\to C_{n-1}$ be the Hermitian dual of $d_{n-1}$.
Recall that for every $T:U\to V$, $\ker T=(\mbox{Im }T^*)^\perp$.
It follows that $\ker d_n+\ker \partial_n=C_n$.
Indeed, $(\ker \partial_n)^\perp=\overline{\mbox{Im }d_{n-1}} <\ker d_n$,
as $d_n d_{n-1}=0$.
We define 
\begin{align*}
C_n^0 &:= \ker d_{n+1} \cap \ker \partial_{n_1},\\
C_n^-&:=\ker d_{n+1} \cap (\ker \partial_{n-1})^\perp,\\
C_n^+&:=\ker \partial_{n-1} \cap (\ker d_{n+1})^\perp.
\end{align*}
and observe that these spaces form an orthogonal decomposition of $C_n$,
\[ C_n=C_n^-\oplus C_n^0 \oplus C_n^+. \]
The first two equations below are now immediate and the last two follow by taking the adjoints of the first ones and shifting the indices:
\begin{align*}
\ker d_{n} & =C_n^- \oplus C_n^0,\\
\ker \partial_n & =C_n^+ \oplus C_n^0,\\
\overline{\mbox{Im }d_n} & =C_{n+1}^-,\\
\overline{\mbox{Im }\partial_n} & =C_{n-1}^+.
\end{align*}
We deduce that the map $d_n:C_{n}\to C_{n+1}$ could be decomposed into
\[ d_n:C_n \twoheadrightarrow C_n^+ \overset{\bar{d}_n}{\longrightarrow} C_{n+1}^- \hookrightarrow C_{n+1}, \]
where $C_n \twoheadrightarrow C_n^+$ is the orthogonal projection, $C_{n+1}^- \hookrightarrow C_{n+1}$ is the inclusion 
and $\bar{d}_n:C_n^+ \to C_{n+1}^-$ is an injective transformation which has a dense image.
We also get a similar decomposition
\[ \partial_n:C_n \twoheadrightarrow C_n^- \overset{\bar{\partial}_n}{\longrightarrow} C_{n-1}^+ \hookrightarrow C_{n-1}, \]
where $\bar{\partial}_n=\bar{d}^*_{n-1}$.
Again, $\bar{\partial}_n$ is an injective transformation which has a dense image.

Next we define the following operators:
\begin{align*}
\Delta_n^+&=\partial_{n+1}d_n,\\
\Delta_n^-&=d_{n-1}\partial_n,\\
\Delta_n&=\Delta_n^-+\Delta_n^+.
\end{align*}
One checks easily that $\Delta_n^+$ and $\Delta_n^-$, hence also $\Delta_n$, are positive self adjoint operators
and that $\Delta_n^+\Delta_n^-=0$ and $\Delta_n^-\Delta_n^+=(\Delta_n^+\Delta_n^-)^*=0$.

Let us recall an obvious observation.
for $T:U\to V$, $\ker T=\ker T^*T$ and $\overline{\mbox{Im } T}=\overline{\mbox{Im }T^*T}$.
Indeed, it is obvious that $\ker T< \ker T^*T$ thus the first statement follows from the fact that for $v\in \ker T^*T$,
$\|Tv\|^2=\langle Tv,Tv \rangle=\langle v,T^*Tv \rangle=0$ while the second statement follows from the first:
\[ \overline{\mbox{Im } T}=(\ker T)^\perp=(\ker T^*T)^\perp=\overline{\mbox{Im } T^*T}. \]
We conclude the following equations.
\begin{align*} %\label{eq:identification}
\ker \Delta_n^+ & =\ker d_{n}=C_n^- \oplus C_n^0,\\
\ker \Delta_n^- &=\ker  \partial_n =C_n^+ \oplus C_n^0,\\
\overline{\mbox{Im }\Delta_n^+}&=\overline{\mbox{Im }\partial_{n+1}}=C_n^+,\\
\overline{\mbox{Im }\Delta_n^-}&=\overline{\mbox{Im }d_{n-1}}=C_n^-.
\end{align*}

We now define 
\begin{align*}
\bar{\Delta}_n^+&=\bar{\partial}_{n+1}\bar{d}_n:C_n^+\to C_n^+,\\
\bar{\Delta}_n^-&=\bar{d}^{n-1}\bar{\partial}_n:C_n^-\to C_n^-,
\end{align*}
and observe that these are injective positive self adjoint operators which have dense images.
We also observe the orthogonal decomposition
\[ \Delta_n=\bar{\Delta}_n^-\oplus 0 \oplus \bar{\Delta}_n^+:C_n^-\oplus C_n^0 \oplus C_n^+ \to C_n^-\oplus C_n^0 \oplus C_n^+ \]
and conclude that
\[ \ker \Delta_n=C_n^0. \]

Associated with $(C_\bullet,d_\bullet)$ one sets, as usual, the $n$'th cohomology to be $H^n=\ker d_n/\mbox{Im } d_{n-1}$,
and the reduced $n$'th cohomology to be $\bar{H}^n=\ker d_{n}/\overline{\mbox{Im } d_{n-1}}$.
One says that the $n$'th cohomology is reduced if $H^n=\bar{H}^n$, that is if the image of $d_{n-1}$ is closed.
This condition is obviously equivalent to $H^n$ being Hausdorff in the quotient topology.

\begin{proposition} \label{prop:cohovanishing}
\begin{enumerate}
\item The reduced $n$'th cohomology $\bar{H}^n$ is isomorphic to $\ker \Delta_n$,
in particular $\bar{H}^n=0$ if and only if  $\ker \Delta_n=0$.
\item $H^n$ is reduced 
if and only if  $\bar{\Delta}_n^-$
is invertible
if and only if  there exists $\epsilon>0$ such that $\Delta_n^-(\Delta_n^--\epsilon)$ is positive.
\item $H^{n+1}$ is reduced 
if and only if  
$\bar{\Delta}_n^+$ is invertible if and only if 
there exists $\epsilon>0$ such that $\Delta_n^+(\Delta_n^+-\epsilon)$ is positive.
\item Both $H^n$ and $H^{n+1}$ are reduced if and only if  
both $\bar{\Delta}_n^-$ and $\bar{\Delta}_n^+$ are invertible
if and only if 
there exists $\epsilon>0$ such that $\Delta_n(\Delta_n-\epsilon)$ is positive.
\item $H^n=0$ and $H^{n+1}$ is reduced if and only if  $\Delta_n$ is invertible 
if and only if  there exists $\epsilon>0$ such that $\Delta_n-\epsilon 1 $ is positive.
\end{enumerate}
\end{proposition}

Before proving the lemma we make two claims.
The first claim is that a positive self adjoint operator $S:V\to V$ is invertible if and only if  there exists $\epsilon>0$ such that the self adjoint operator $S-\epsilon$ is positive,
and in case $S$ is injective this happens
if and only if  there exists $\epsilon>0$ such that the self adjoint operator $S(S-\epsilon)$ is positive.
This claim follows easily by spectral theory.
The second claim is that
for a transformation $T:U\to V$, we have that 
$T$ is an isomorphism if and only if  $T^*$ is an isomorphism if and only if  $T^*T$ is invertible.
It is indeed clear that $T$ is an isomorphism if and only if  $T^*$ is an isomorphism and this also implies that $T^*T$ is invertible.
To see the remaining implication, assume that $T^*T$ is invertible
and deduce from the equation $\overline{\mbox{Im } T}=\overline{\mbox{Im }T^*T}$
that $T$ has a dense image, thus $T^*$ is injective. But the fact that $T^*T$ is invertible also implies that $T^*$ is surjective, thus it is an isomorphism by the open mapping theorem. This proves the claim.

\begin{proof}
To see (1) note that by the equations $\ker d_{n}=C_n^0\oplus C_n^-$
and $\overline{\mbox{Im }d_{n-1}}=C_n^-$ we have that 
$\bar{H}^n\simeq C_0$ and note that $\ker \Delta_n=C_0$.

For (2) note that 
$H^n$ is reduced if and only if  ${d}_{n-1}$ is onto $\overline{\mbox{Im }d_{n-1}}=C_n^-$ if and only if  $\bar{d}_{n-1}$ is surjective
if and only if 
$\bar{d}_{n-1}$ is an isomorphism (by the open mapping theorem) if and only if 
$\bar{\Delta}_n^-$ is invertible (by the second claim)
if and only if  there exists $\epsilon>0$ such that $\bar{\Delta}_n^-(\bar{\Delta}_n^--\epsilon)$ is positive (by the first claim)
if and only if  there exists $\epsilon>0$ such that $\Delta_n^-(\Delta_n^--\epsilon)$ is positive (as this operator decomposes on $C_n^-\oplus C_n^0 \oplus C_n^+$ as $\bar{\Delta}_n^-(\bar{\Delta}_n^--\epsilon)\oplus 0 \oplus 0$).

Similarly, for (3) we have that 
$H^{n+1}$ is reduced if and only if  $d_n$ is onto $\overline{\mbox{Im }d_{n}}=C_{n+1}^-$
if and only if  $\bar{d}_{n}$ is surjective if and only if  $\bar{d}_{n}$ is an isomorphism (by the open mapping theorem) if and only if  
$\bar{\Delta}_n^+$ is invertible (by the second claim)
if and only if  there exists $\epsilon>0$ such that $\bar{\Delta}_n^+(\bar{\Delta}_n^+-\epsilon)$ is positive (by the first claim)
if and only if  there exists $\epsilon>0$ such that $\Delta_n^+(\Delta_n^+-\epsilon)$ is positive
(as this operator decomposes on $C_n^-\oplus C_n^0 \oplus C_n^+$ as $0 \oplus 0 \oplus \bar{\Delta}_n^+(\bar{\Delta}_n^+-\epsilon)$).

(4) follows from (2) and (3) given that 
$\Delta_n^+\Delta_n^-=\Delta_n^-\Delta_n^+=0$.

To see (5) note that by the first claim
$\Delta_n$ is invertible 
if and only if  there exists $\epsilon>0$ such that $\Delta_n-\epsilon$ is positive.
Assume that $\Delta_n$ is invertible.
Then in particular $C_0=\ker \Delta_n=0$, thus $C_n= C_n^-\oplus C_n^+$ and $\Delta_n$ decomposes as $\bar{\Delta}_n^-\oplus \bar{\Delta}_n^+$. As $\Delta_n$ is invertible, it follows that both  $\bar{\Delta}_n^-$ and $\bar{\Delta}_n^+$
are invertible. By (4) we have that both $H^n$ and $H^{n+1}$
are reduced and we conclude that $H^n=\bar{H}^n=0$ by (1).
Thus indeed, $H^n=0$ and $H^{n+1}$
is reduced.
Assume now that $H^n=0$ and $H^{n+1}$ is reduced.
Then
by (1), $C_n=C_n^-\oplus C_n^+$
and $\Delta_n$ decomposes as  $\bar{\Delta}_n^-\oplus \bar{\Delta}_n^+$.
Since both $H^n$ and $H^{n+1}$ are reduced,
by (4)
we have that both $\bar{\Delta}_n^-$
and $\bar{\Delta}_n^+$ are invertible
and we conclude that $\Delta_n$ is invertible.
\end{proof}

\begin{remark} \label{rem:epsilonrational}
Note that in Proposition~\ref{prop:cohovanishing}, by the positivity of the operators $\Delta^-$ in (2),
$\Delta^+$ in (3),
$\Delta$ in (4)
and $1$ in (5),
the parameter $\epsilon$ in each case could be chosen to be a positive rational.
\end{remark}

%%%%%%

\section{The main Theorem} \label{sec:main}

\subsection{Proof of Proposition~\ref{prop:D_n}} \label{subsec:D_n}

Let $\Gamma$ be a group acting by automorphisms on a simplicial complex $X$.
We denote by $X^{(n)}$ the $n$'th skeleton of $X$, that is the set of $n$-dimensional oriented (that is, ordered) simplices in $X$.
Let $V$ be a $\mathbb{Q}\Gamma$-module.
We denote by $C_n=C_n(V)$ the vector space of alternating $\Gamma$-equivariant maps from $X^{(n)}$ to $V$.
For $t=0,\ldots,n+1$, we let $f_n^t:C_n\to C_{n+1}$ be the $t$'th face map,
that is for $\phi\in C_{n+1}$ and $\sigma=(\sigma_0,\ldots,\sigma_{n+1})\in X^{(n+1)}$, the evaluation of $f_n^t(\phi)$ on $\sigma$
is given by $\phi(\sigma_0,\ldots,\hat{\sigma}_t,\ldots,\sigma_{n+1})$,
and we let 
\[ d_n=\sum_{t=0}^n (-1)^t f_n^t:C_n\to C_{n+1} \]
be the standard boundary map.
The following is a well known fact.

\begin{theorem} \label{thm:groupcoho}
If $X$ is contractible and the $\Gamma$-stabilizers are finite then the chain complex $(C_\bullet,d_\bullet)$ computes the cohomology of $\Gamma$ with coefficients in $V$,
that is $H^n(C_\bullet,d_\bullet)\simeq H^n(\Gamma,V)$.
\end{theorem}

The chain complex $(C_\bullet,d_\bullet)$ is called the equivariant chain complex associated with the action of $\Gamma$ on $X$.
Upon choosing a fundamental domain, one can introduce a non-equivariant, isomorphic chain complex.
This is what we describe next.
For every $n$ we let $Y^{(n)}\subset X^{(n)}$ be a fixed fundamental domain for the $\Gamma$-action on $X^{(n)}$
and we let $\bar{C}_n$ be the vector space of all maps from $Y^{(n)}$ to $V$.
The restriction map $C_n \to \bar{C}_n$ is a linear isomorphism.
Conjugating the boundary maps with these isomorphisms, we get a new, but isomorphic, chain complex
\[ \cdots \to \bar{C}_{n-1} \overset{\bar{d}_{n-1}}{\longrightarrow} \bar{C}_n \overset{\bar{d}_{n}}{\longrightarrow} \bar{C}_{n+1} \to \cdots \]
In particular, for every $n$, $H^n(\bar{C}_\bullet,\bar{d}_\bullet)\simeq H^n(C_\bullet,d_\bullet)$.

We will describe more explicitly the boundary maps $\bar{d}_n$ appearing above.
For $\phi\in C_{n}$ and $\sigma=(\sigma_0,\ldots,\sigma_{n+1})\in Y^{(n+1)}$,
finding $\tau\in Y^{(n)}$ and $g\in \Gamma$ such that $g\tau$ equals the $t$'th face of $\sigma$, namely $(\sigma_0,\ldots,\hat{\sigma}_t,\ldots,\sigma_{n+1})$, we have
\[ f_n^t(\phi)(\sigma)=\phi(g\tau)=g\phi(\tau). \]
We thus define $\bar{f}_n^t:\bar{C}_{n}\to \bar{C}_{n+1}$ by
\begin{equation} \label{eq:sigmatau}
\bar{f}_n^t(\bar{\phi})(\sigma)=g\bar{\phi}(\tau),
\end{equation}
for $\bar{\phi}\in \bar{C}_n$ and $\sigma$, $\tau$ and $g$ as above.
We thus get the explicit description $\bar{d}_n=\sum_{t=0}^n (-1)^t \bar{f}_n^t$.

We now fix a natural $N$ and assume that the for every $n\leq N+1$, $X^{(n)}$ has finitely many $\Gamma$-orbits,
equivalently $Y^{(n)}$ is finite.
We set $k_n=|Y^{(n)}|$.
We enumerate the elements of the sets $Y^{(n)}$, thus identify $\bar{C}_n$ with $V^{k_n}$.
Under these identifications, the map $\bar{f}_n^t$ is represented by an $k_n\times k_{n+1}$-matrix with coefficients in $\mathbb{Q}\Gamma$ 
as follows:
if $\sigma\in Y^{(n+1)}$ and $\tau\in Y^{(n)}$ are enumerated $i$ and $j$ correspondingly, then the entry of this matrix at $(i,j)$ will be $g$
if $g\tau$ is the $t$'th face of $\sigma$, as in equation~(\ref{eq:sigmatau}), and will be 0 otherwise.
Accordingly, $\bar{d}_n=\sum_{t=0}^n (-1)^t \bar{f}_n^t$ is represented by a matrix $D_n\in M_{k_n\times k_{n+1}}(\mathbb{Q}\Gamma)$.
Note that the matrices $D_n\in M_{k_n\times k_{n+1}}(\mathbb{Q}\Gamma)$, which do depend on the action of $\Gamma$ on $X$
and the choice of fundamental domains $Y^{(n)}$, are independent of the representation $V$.
Also check that the relation $D_nD_{n-1}=0\in M_{k_{n-1}\times k_{n+1}}(\mathbb{Q}\Gamma)$ holds for every $n\leq N$.
It follows that the cohomology groups $H^n(\bar{C}_\bullet,\bar{d}_\bullet)$ are isomorphic to the cohomology groups of the complex
\[ \cdots \to V^{k_{n-1}} \overset{D_{n-1}}{\longrightarrow} V^{k_n} \overset{D_{n}}{\longrightarrow} V^{k_{n+1}} \to \cdots \]
In view of the isomorphism $H^n(\bar{C}_\bullet,\bar{d}_\bullet)\simeq H^n(C_\bullet,d_\bullet)$,
Proposition~\ref{prop:D_n} now follows from Theorem~\ref{thm:groupcoho}.

\subsection{Proof of The Main Theorem} \label{subsec:proof}

We will show that both (1) and (2) in The Main Theorem are equivalent to

\begin{enumerate}
\setcounter{enumi}{2}
\item
For every unitary representation $\rho$ of $\Gamma$, the image of $\Delta_n$ under $M_k(\rho)$ is invertible.
\end{enumerate}

We note that by Lemma~\ref{lem:squarematrices} we have 
$D_n^*D_n,~D_{n-1}D_{n-1}^* \in M_{k_n}(\mathbb{Q}\Gamma)_+$, hence also 
$\Delta_n=D_n^*D_n+D_{n-1}D_{n-1}^* \in M_{k_n}(\mathbb{Q}\Gamma)_+$.
Thus (3) is equivalent to (2) by Corollary~\ref{cor:unitaryinvcriterion}.
By Proposition~\ref{prop:D_n}, for every unitary representation $\rho$ of $\Gamma$ on a Hilbert space $V$, the cohomology groups $H^n(\Gamma,\rho)$ are isomorphic to the cohomology groups of the complex
\[ \cdots \to V^{k_{n-1}} \overset{D_{n-1}}{\longrightarrow} V^{i_n} \overset{D_{n}}{\longrightarrow} V^{k_{n+1}} \to \cdots \]
Thus (3) is also equivalent to (1) by Proposition~\ref{prop:cohovanishing}(5).
We conclude that indeed (1) is equivalent to (2).

%%%%%

\subsection{Some further remarks} \label{subsec:remarks}

Property $(T)$ has several different characterizations.
In the introduction we mentioned one of them:
a finitely generated group $\Gamma$ has property $(T)$ if and only if 
$H^1(\Gamma,\rho)=0$ for every unitary representation $\rho$.
In fact, in \cite{ozawa} Ozawa used another, equivalent, characterization:
a finitely generated group $\Gamma$ has property $(T)$ if and only if 
the cohomology $H^1(\Gamma,\rho)$ is \emph{reduced} for every unitary representation $\rho$.
In the course of his proof he used Proposition~\ref{prop:cohovanishing}(3) above in the special case $n=0$
to establish that, in a certain setting involving $\Delta_0$, all first cohomology groups are indeed reduced.
This is particularly desirable, as in this setting $\Delta_0$ is the classical Laplacian element in the group algebra $\mathbb{Q}\Gamma$,
rather then an element in a matrix group over it, as appears in our considerations here for $\Delta_n$, $n>0$.

When trying to generalize Ozawa's work, one should take notice of the fact that in higher degrees, the following two properties of a group $\Gamma$
are not equivalent in general.

\begin{enumerate}
\item $H^n(\Gamma,\rho)=0$ for every unitary representation $\rho$.
\item $H^n(\Gamma,\rho)$ is reduced for every unitary representation $\rho$.
\end{enumerate}

The reader should be aware that 
property (1) was taken as the definition of $n$-Kazhdan group in \cite[Definition~4.1]{chifre-etal} while
property (2) was taken as the definition of $(T_{n-1})$ group in \cite[Definition~30]{bader-nowak}.
Clearly, property (1) implies property (2).
As remarked before, these properties are equivalent for $n=1$.
The fact that these properties are not equivalent in general for $n\geq 2$ is illustrated by the following proposition. 

\begin{proposition}[Dymara-Januszkiewicz] \label{A2}
Fix an integer $n\geq 2$ and a sufficiently large prime $p$.
Let $\Gamma$ be a lattice in $PGL_{n+1}(\mathbb{Q}_p)$.
Then $\Gamma$ satisfies property (2) but it does not satisfy property (1).
\end{proposition}

We stress that Proposition~\ref{prop:cohovanishing}
gives a natural setting for establishing property (2) for various groups.
Indeed, Proposition~\ref{A2} is supported by a computer assisted proof provided to our request by Marek Kaluba for some specific instances of groups
$\Gamma$, using Proposition~\ref{prop:cohovanishing}(2).
These instances are the four groups acting simply transitively on the set of chambers of an $\tilde{A}_2$ building of thickness 3 
constructed by Tits and Ronan in \cite[\S3.1]{tits} and \cite[Theorem~2.5]{ronan}.

We note however that Proposition~\ref{prop:cohovanishing} does not give a checkable ``if and only if'' criterion for property (2).
In contrast, the work \cite{ozawa} did provide such a criterion for $n=1$.
There, following \cite{netzer-thom}, a use was made of the fact that $\Delta=\Delta_0$ is an order unit
in the augmantation ideal of the group algebra.
We are not aware of an analogue of this fact in higher degrees.

We will end this paper by proving Proposition~\ref{A2}.
This result is essentially due to Dymara and Januszkiewicz, taking into account the topological Shapiro's Lemma \cite[Theorem~8.7]{blanc}.
We are grateful to Jan Dymara for discussing it with us.

\begin{proof}
We will use \cite{dymara-janushkevich}.
We set $G=PGL_{n+1}(\mathbb{Q}_p)$ and let $X$ be its Bruat-Tits building.
The assumption on the prime $p$ is so that the thickness of $X$ is sufficiently large 
in the sense of \cite[Theorem~A]{dymara-janushkevich} so that $(X,G)$ is in the class $\mathcal{B}+$.
By \cite[Theorem~D]{dymara-janushkevich} we have that $G$ satisfies property (2),
thus it follows by the topological Shapiro's Lemma \cite[Theorem~8.7]{blanc} that also $\Gamma$ satisfies property (2).

We are left to show that $\Gamma$ does not satisfy property (1).
By Shapiro's lemma, it is enough to show that some finite index subgroup of $\Gamma$
does not satisfy property (1). We will argue to show that in fact, for some finite index subgroup $\Gamma'<\Gamma$, 
$H^n(\Gamma',\mathbb{C})\neq 0$.
We assume this is not the case, that is $H^n(\Gamma',\mathbb{C})= 0$ for every finite index subgroup $\Gamma'<\Gamma$,
and argue by contradiction.
Note that for such a finite index subgroup $\Gamma'$, we have $\dim(H^0(\Gamma',\mathbb{C}))=1$
while $H^i(\Gamma',\mathbb{C})=0$ for every $i>0$.
For $0<i<n$ this follows by Shapiro's Lemma
\cite[Theorem~8.7]{blanc} and \cite[Theorem~B]{dymara-janushkevich}, for $i=n$ this follows by our negation assumption and for 
$i>n$ this follows by Theorem~\ref{thm:groupcoho}, as $X$ is a contractible $n$-dimensional simplicial complex.
If $\Gamma'$ is torsion free, it follows that its Euler characteristic equals 1, as it is given by an alternating sum of Betti numbers. 
As $\Gamma$ is a finitely generated linear group, thus residually finite, it has torsion free subgroups of arbitrary large finite index.
As the Euler characteristic is proportional to the index, we get a contradiction.
Thus indeed $\Gamma$ does not satisfy property (1).
\end{proof}

\begin{remark}\normalfont
In fact, the assumption that $p$ is sufficiently large in Proposition~\ref{A2} is superfluous, and the proposition is valid for every prime $p$.
That $\Gamma$ does not satisfy property (1) follows by \cite{oppenheim-gr}.
Also that $\Gamma$ satisfies property (2) follows from by the same method of \cite{oppenheim-gr}, however this is not currently written in a way which is easy to cite.
We note that by removing the assumption on the thickness, one could obtain an analogue of Proposition~\ref{A2} which is valid for every simple group of 
higher rank over a non-archimedean local field.
\end{remark}

%%%%%%%

\end{document}